\documentclass[12pt]{amsart}

\usepackage{amssymb}
\usepackage{amsmath}
\usepackage{amsthm}
\usepackage{times}
\usepackage{graphicx}
 \usepackage[curve,all]{xy}\textheight22truecm

\DeclareGraphicsExtensions{.png,.jpg,.jpeg}

\usepackage{color}
\newtheorem{theorem}{Theorem}
\newtheorem{lemma}{Lemma}

\newtheorem{proposition}{Proposition}
\newtheorem{remark}{Remark}

\usepackage[curve,all]{xy}
 \xyoption{curve}
 \xyoption{color}
 \xyoption{line}
\xyoption{arc}

\usepackage[
            pdftex,
           colorlinks=true,
            urlcolor=blue,       
            filecolor=blue,     
            linkcolor=blue,       
            citecolor=blue,         
            pdftitle= {Title},
            pdfauthor={Author},
            pdfsubject={Subject},
            pdfkeywords={Key}
            pagebackref,
            pdfpagemode=None,
            bookmarksopen=true]
            {hyperref}
\usepackage{hyperref}
\usepackage{color}

\newcommand{\calO}{\mathcal{O}}

\newcommand{\bbB}{\mathbb{B}}

\newcommand{\bbP}{\mathbb{P}}

\newcommand{\bbQ}{\mathbb{Q}}

\newcommand{\bbZ}{\mathbb{Z}}

\newcommand{\bfc}{\mathbf{c}}

\newcommand{\bfe}{\mathbf{e}}

\newcommand{\bfr}{\mathbf{r}}

\newcommand{\Aut}{\textup{Aut}}

\newcommand{\Pic}{\textup{Pic}}

\newcommand{\rk}{\textup{rk}}

\newcommand{\SU}{\textup{SU}}

\newcommand{\NS}{\textup{NS}}

\newcommand{\Tr}{\textup{Tr}}

\renewcommand\emptyset\varnothing

\newcommand{\beq}{\begin{equation}}
\newcommand{\eeq}{\end{equation}}

\date{\today}
\title{The bicanonical map of the Cartwright-Steger surface}
\author{JongHae Keum}

\begin{document}

\maketitle
\begin{abstract} We prove that the bicanonical map of the Cartwright-Steger surface is an embedding. We also discuss two minimal surfaces of general type, both covered by the Cartwright-Steger surface. One has $K^2=2, p_g=1, \pi_1=\{1\}$ and the other has $K^2=1, p_g=0, \pi_1=\bbZ/2\bbZ$.
\end{abstract}

The Cartwright-Steger surface is a minimal surface of general type with $p_g=q=1$ and $K^2=3c_2=9$. It is an arithmetic ball quotient found by Cartwright and Steger \cite{CS}, \cite{CS2}. Reider's theorem \cite{reider} implies that the bicanonical system of a ball quotient $X$ is base point free, thus defines a morphism.

Let $X$ be the Cartwright-Steger surface and
$$\Phi_{2,X} : X \to {\bbP}^9$$ be the bicanonical map. This map may fail to separate points only on certain curves, as specified in  the criterion of Reider's theorem \cite{reider}. In this note we prove that such curves do exist on the Cartwright-Steger surface.

\begin{theorem}\label{Main} On the Cartwright-Steger surface $X$ there exists no effective curve $B$ such that either $3B$ is numerically equivalent to the canonical class $K_X$ or $K_XB=2$ and $B^2=0$. In particular, the bicanonical map $\Phi_{2,X}$ is an embedding into ${\bbP}^9$.
\end{theorem}

  It is known that Cartwright-Steger surface $X$ has automorphism group $\Aut(X)=\bbZ/3\bbZ$ and the quotient is simply connected \cite{CS2}, and the action has only isolated fixed points, three of type $1/3(1,1)$ and six of type $1/3(1,2)$  \cite{CKY}. The latter was also obtained by geometric arguments by F. Catanese, T. Domingo, M. Stover and the author, and by I. Dolgachev. It follows that the minimal resolution $Y$ of the quotient $X/\Aut(X)$ is a minimal surface of general type with
  $$K_Y^2=2,\,\,\, p_g(Y):=h^{2,0}(Y)=1,\,\,\, \pi_1(Y)=\{1\}.$$

  L. Borisov found an involution $\alpha$ on $Y$, as a bi-product of his computation of the equation of $Y$. This automorphism does not come from the ball. It turns out to be fixed point free and the quotient $Z=Y/\langle \alpha\rangle$ is a minimal surface of general type with $$K_Z^2=1,\,\,\, p_g(Z)=0,\,\,\, \pi_1(Z)=\bbZ/2\bbZ.$$

\section{Known facts on the  Cartwright-Steger surface}

We collect known facts on the  Cartwright-Steger surface \cite{CS}, \cite{CS2}, \cite{CKY}, \cite{DM}, \cite{D}.
Let
$$\pi:X = \bbB/\Pi\to \bbP(1,2,3) = \bbB/\bar{\Gamma}$$
be the natural projection map of degree $[\bar{\Gamma}:\Pi] = 3.288.$ The Deligne-Mostow quotient $\bbP(1,2,3) = \bbB/\bar{\Gamma}$ contains a line $D_B$ and a cuspidal curve $D_A$. The curve $D_B$ (resp.$D_A$) is the image of the set of mirrors of complex reflections of order 4 (resp. 3) in $\bar{\Gamma}$. See p.111, \cite{DM}. Let $P_1$ be the cusp of $D_A$,  $P_2$ the tangential intersection point of $D_A$ and $D_B$, $P_3$ the transversal intersection point of $D_A$ and $D_B$, and $P_4, P_5$ be the singular points of type $1/2(1,1)$ and $1/3(1,2)$ respectively. We know that $D_A\cap D_B=\{P_2, P_3\}$, $P_4\in D_B$ and $P_5\notin D_A\cup D_B$. The pre-image $\pi^{-1}(P_2)$  consists of three points $O_1,O_2,O_3$. From  \cite{D} we also know that
$$\pi^{-1}(D_B) = E_1+E_2+E_3,$$
where $E_1,E_2,E_3$ are irreducible curves of geometric genus 4. As a reducible curve, $\pi^{-1}(D_B)$ has 6 transversal branches at each $O_i$, and is smooth elsewhere.  Thus $E_1,E_2,E_3$ intersect each other only at $O_1,O_2,O_3$. Their multiplicities at $O_i$ are given in Table \ref{multEC}. The intersection number $E_iE_j$ is equal to the dot product of their multiplicity vectors if $i\neq j,$ and \beq\label{E2} E_j^2=1-g(E_j)+\sum_{i} m(E_j, O_i)[m(E_j, O_i)-1].\eeq
 See Table \ref{intCS}. Here we use the fact that $K_XE=3g(E)-3$  for a totally geodesic curve $E$ of geometric genus $g(E)$.

\begin{table}[ht]
\centering
  \begin{tabular}{|c|c|c|c|c|c|}\hline
&$O_1$&$O_2$&$O_3$&$g$&$p_a$  \\ \hline\hline
$E_1$&3&1&2&4&8\\
$E_2$&2&1&3&4&8\\
$E_3$&1&4&1&4&10\\
$C_1,C_2$&0&1&2&4&5\\
$C_3,C_4$&4&3&2&10&\\ \hline \hline
\end{tabular}
\caption{}\label{multEC}
\end{table}


There are two curves $C_1,C_2$ of geometric genus 4, and two curves $C_3,C_4$ of geometric genus 10 such that
$$\pi^{-1}(D_A) = C_1+C_2+C_3+C_4.$$
As a reducible curve, $\pi^{-1}(D_A)$ has 8 transversal branches at each $O_i$.
The preimage $\pi^{-1}(P_1)$ consists of 36 points and $\pi^{-1}(P_3)$ 72 points. The curves $C_i$ intersect each other at $O_i$, but also intersect transversally at points in $\pi^{-1}(P_1)$, and nowhere else. Table \ref{multEC} contains the multiplicities  at $O_1,O_2,O_3$ of the curves $C_i$.

The information on the curves $E_i, C_j$ are obtained from \cite{DM} and \cite{D}.

\subsection{Intersection numbers of the geodesic curves}
Through discussion with F. Catanese, M. Stover, D. Toledo, we have obtained the intersection numbers of curves $E_i,C_j$, as given in Table \ref{intCS}.
 This table confirms the computation of \cite{CKY}, where the entries involving $C_3, C_4$ are not given explicitly.

\begin{table}[ht]
\centering
  \begin{tabular}{|c|c|c|c|c|c|c|l|}\hline
&$E_1$&$E_2$&$E_3$&$C_1$&$C_2$&$C_3$&$C_4$  \\ \hline\hline
$E_1$&5&13&9&11&11&25&25\\
$E_2$&13&5&9&7&7&29&29\\
$E_3$&9&9&9&9&9&27&27\\
$C_1$&11&7&9&$-1$&17&37 &19 \\
$C_2$&11&7&9&17&$-1$&19 &37 \\
$C_3$&25&29&27&37&19 &71&89 \\
$C_4$&25&29&27& 19&37 &89 &71
\\\hline \hline
\end{tabular}
\caption{}\label{intCS}
\end{table}
In particular $p_a(C_3)=p_a(C_4)=50,$ filling up Table \ref{multEC}.

\subsection{The N\'eron -Severi group $\NS(X)$}
 By Cartwright and Steger \cite{CS}, $$H_1(X, \mathbb{Z})=\mathbb{Z}^2.$$  So by the universal coefficient theorem, $H^2(X, \mathbb{Z})=\mathbb{Z}^5$, torsion free. The three curves, $E_1, E_3, C_1$, are  numerically independent, so $\NS(X)$ is a free group of rank 3. Note that $Pic(X)$  contains torsions, namely the torsion elements of the Picard variety $Pic^0(X)$, an elliptic curve in this case.

\subsection{Fix$(\sigma)$} It is known that $\Aut(X)\cong \bbZ/3\bbZ$ \cite{CS2}.
The order 3 automorphism $\sigma$ of $X$ fixes 9 points,
\beq
Fix(\sigma)=\{O_1, O_2, O_3, Q_1,\ldots, Q_6\},
\eeq
where
$O_i$ is of type $1/3(1,1)$ and $Q_i$ of type  $1/3(1,2)$. The points $Q_1, ...,Q_6$ lie over the the singular point of type  $1/3(1,2)$ of the Deligne-Mostow quotient $\bbP(1,2,3)$, hence for any $i$
\beq
Q_i\notin \pi^{-1}(D_B)\cup\pi^{-1}(D_A).
\eeq  The induced action of $\sigma$ on the elliptic curve $Alb(X)$ fixes 3 points. Let $F_0, F_1, F_2$ be the 3 $\sigma$-invariant Albanese fibres. One may assume that
\beq
O_1, O_2, O_3\in F_0, \,\,\,Q_1, Q_2, Q_3\in F_1,\,\,\,Q_4, Q_5, Q_6\in F_2.
\eeq

\subsection{The action of $\sigma$ on $\pi^{-1}(D_B)$ and $\pi^{-1}(D_A)$}

\begin{itemize}
\item Each $E_i, C_j$ is  $\sigma$-invariant, i.e. $\sigma(E_i)=E_i, \,\,i=1,2,3.$ and $\sigma(C_j)=C_j, \,\,j=1,2,3,4.$
\end{itemize}

This follows from the fact that both $\pi^{-1}(D_A)$ and  $\pi^{-1}(D_B)$ have transversal branches at $O_i$.
Indeed, locally at $O_i$, $\sigma(x, y)=(\zeta x,\zeta y)$, $\zeta$ is a third root of 1, hence preserves the tangent line of every branch.

\begin{proposition} $\Aut(X)$ is cohomologically trivial, i.e., acts as the identity on $H^2(X, \mathbb{Z})=\mathbb{Z}^5$.
\end{proposition}

\begin{proof} Since the classes of the three curves, $E_1, E_3, C_1$, generate $\NS(X)\otimes \bbQ$,  $\Aut(X)$ acts as the identity on $\NS(X)\otimes \bbQ$. By the holomorphic Lefschetz fixed point formula (cf. \cite{K12}), from the information on the fixed locus of $\Aut(X)$ one sees that  $\sigma^*$ acts on $H^2(X, \calO_X)$ as the identity, and on $H^1(X, \calO_X)$ as the multiplication by a third root of unity.
It follows that $\Aut(X)$ acts as the identity on $H^2(X, \mathbb{Q})$, hence on $H^2(X, \mathbb{Z})$ since the latter has no torsion element.
\end{proof}

\begin{remark} $\sigma^*$ acts on $H^1(X, \bbZ)=\bbZ^2$ nontrivially, namely as $\sigma^*(v_1)=v_2$, $\sigma^*(v_2)=-v_1-v_2$ with respect to a suitable basis $v_1, v_2$.
Thus in the theorem of \cite{peters} the condition that $|K_X|$ is base point free  is necessary.
\end{remark}

\subsection{Linear equivalences among $\Aut(X)$-invariant curves}
Since the quotient $X/\Aut(X)$ is simply connected \cite{CS2}, any numerical equivalence among $\Aut(X)$-invariant curves is indeed a linear equivalence modulo an $\Aut(X)$-invariant divisor class $\in \Pic^0(X)$ (note that the $\Aut(X)$-action on $\Pic^0(X)$ fixes 3 elements, that form a subgroup of order 3.) Modulo $\Pic^0(X)^{\Aut(X)}$,
$$K_X\equiv E_3,$$
$$E_1+E_2\equiv 2E_3,$$
$$4E_3\equiv C_1+C_3\equiv C_2+C_4,$$
$$3E_3\equiv E_1+C_1+C_2,$$
$$E_2+E_3\equiv C_1+C_2.$$
 In particular, each of the above equivalences is a linear equivalence, once multiplied by 3, e.g., $3(E_1+E_2)\equiv 6E_3.$

\begin{remark} It is not clear if the above equivalences are indeed  linear equivalences, without being  multiplied by 3. (Remark 5.7 in \cite{CKY} needs proof or corrections. On the simply connected quotient
$X/\Aut(X)$  a numerical equivalence between Cartier divisors is a linear equivalence, so is its pull back to $X$, but this may not hold for $\bbQ$-Cartier divisors.) In general, if a compact complex surface $V$ admits a $\bbZ_m$-action with only isolated fixed points such that $V/\bbZ_m$ is simply connected, and if two $\bbZ_m$-invariant effective
curves $A$ and $B$ on $V$ are numerically equivalent, then $mA$ and $mB$ are linearly equivalent. But $A$ and $B$ may not be linearly equivalent, as there are examples, e.g., consider a product of two elliptic curves and its Kummer surface.
\end{remark}

\subsection{The Albanese map} Every fibre of the Albanese map $\alpha :X\to Alb(X)$ is irreducible and reduced. Let $F$ be a general smooth fibre. By \cite{CKY} $$g(F)=19$$ and $F$ is numerically equivalent to $-E_1+5E_2$.

\section{First step of Proof}

First we recall Reider's theorem \cite{reider} by stating the expanded version given in Theorem 11.4 of \cite{bpv}.

\begin{theorem}\label{reider} \cite{reider} Let L be nef divisor on a smooth projective surface $X$.
\begin{enumerate}
\item Assume that $L^2\ge 5$. If $P$ is a base point of the linear system $|K_X+L|$, then $P$ lies on an effective divisor $B$ such that
\begin{enumerate}
\item $BL=0,\,\,\,\, B^2=-1$, or
\item $BL=1,\,\,\,\, B^2=0$.
\end{enumerate}
\item
Assume that $L^2\ge 9$. If $P$ and $Q$, possibly infinitely near,  are not base points of $|K_X+L|$ and fail to be separated by $|K_X+L|$, then they lie on an effective curve $B$,  depending on $P,\,\,Q$, satisfying one of the following:
\begin{enumerate}
\item $BL=0,\,\,\,\, B^2=-2$ or $-1$;
\item $BL=1,\,\,\,\, B^2=-1$ or $0$;
\item $BL=2,\,\,\,\, B^2=0$;
\item $L^2=9$ and $L$ is numerically equivalent to $3B$.
\end{enumerate}
\end{enumerate}
\end{theorem}

A ball quotient cannot contain a curve of geometric genus $0$ or $1$.
Applying Reider's theorem to $L=K_X$, one sees that the bicanonical system of a ball quotient is base point free, thus defines a morphism. Moreover the bicanonical system is very ample unless the surface contains an effective divisor with the property $(2c)$ or $(2d)$.

 Consider the case  $(2d)$. Suppose that $K_X$ is numerically equivalent to $3B$. Since $H^2(X, \bbZ)$ is torsion-free,
 $$K_X\equiv 3B+t$$ for some $t\in Pic^0(X)$, where "$\equiv$" is linear equivalence. Since $Pic^0(X)$ is a divisible group, one can write $t=3t'$ in $Pic^0(X)$, thus see that
 $K_X$ is divisible by 3 in $Pic(X)$. The 3-divisibility of $K_X$ is equivalent to the liftability of the fundamental group $\pi_1(X)$ to $\SU(2,1)$ \cite{Ko}. But the explicit computation of the fundamental group by \cite{CS2} shows that $\pi_1(X)$ does not lift to $\SU(2,1)$. This rules out the possibility $(2d)$.

\medskip
 It remains to consider the case  $(2c)$.

 \begin{lemma}\label{smoothD} Suppose that there is an effective divisor $B$ on $X$ with $BK_X=2,\,\, B^2=0$. Then $B$ is an irreducible smooth curve of genus $2$.
\end{lemma}

\begin{proof}
 Suppose that there is such an effective divisor $B$. If $B$ is reducible, then since $K_X$ is ample, there is an irreducible component  $B_1$ of $B$ with $B_1^2\le 0$, $B_1K_X=1$, impossible. If $B$ is irreducible and singular, then $B$ has geometric genus $\le 1$, again impossible, since a ball quotient cannot contain a curve of geometric genus $0$ or $1$.
\end{proof}

\section{Key Lemma}

The following lemma will play a key role in our proof of Theorem \ref{Main}.

\begin{lemma}\label{Main2} Suppose that the Cartwright-Steger surface $X$ contains a smooth curve $B$ with $K_XB=2, B^2=0$. Then the following hold.
 \begin{enumerate}
 \item All such curves $B$ define the same class in the N\'eron-Severi group $\NS(X)\subset H^2(X, \bbZ)$. In other words, the difference of two such curves is an element of the Picard variety $\Pic^0(X)$ which is an elliptic curve.
 \item The image $\sigma(B)$ under an automorphism $\sigma$ is another such curve, and is disjoint from $B$, where $\sigma$ is a generator of $\Aut(X)\cong \bbZ_3$. The three curves $B$, $\sigma(B)$, $\sigma^2(B)$ are mutually disjoint.
 \end{enumerate}
\end{lemma}

\begin{remark}  It can be shown that $mB$ does not move in an algebraic family for $m<9$. Thus $\sigma(B)-B$, being an element of $\Pic^0(X)$, has order at least 9, may have infinite order. In particular, $\sigma(B)-B$ is not a 3-torsion.
The non-existence of such a curve $B$ is a subtle problem. The rest of this paper will be devoted to its proof.
\end{remark}

On the Cartwright-Steger surface $X$ the three curves $E_1, E_3, C_1$ form a $\mathbb{Q}$-basis for $\NS(X)$ with  intersection matrix
\begin{displaymath}
\left( \begin{array}{ccc}
5 & 9 & 11 \\
9 & 9 & 9 \\
11 & 9 & -1 \\
\end{array} \right)
\end{displaymath} with determinant $18^2$.
Any divisor $D$ on $X$ with
 $$DE_1=a,\quad DE_3=b, \quad DC_1=c$$ must be expressed as
 \beq\label{eqD} D\sim xE_1+yE_3+zC_1\eeq  where
$$x= \frac{-5a+6b-c}{18}$$
$$y= \frac{6a-7b+3c}{18}$$
$$z= \frac{-a+3b-2c}{18}.$$
The equality
$$D^2= D( xE_1+yE_3+zC_1)= xa+yb+zc$$ becomes
\beq\label{eqc} 2c^2-2(3b-a)c+5a^2+7b^2-12ab+18D^2=0.\eeq
The discriminant of this quadratic equation for $c$
$$\frac{\Delta}{4} = -(3a-3b)^2+4b^2-36D^2$$  must be a square number, since $a, b, c$ are integers.
In particular
 \beq\label{eqDelta} 4b^2-36D^2=(3a-3b)^2+s^2.\eeq for some integer $s\ge 0$.

 \bigskip
{\bf Step I.}
\textit{ Suppose that $X$ contains a divisor $D$, not necessarily effective,  with $K_XD=2, D^2=0$. Then $D$ is numerically equivalent to either
$$D_I= 1/9(E_1-E_3+2C_1)\sim 1/9(2E_3+C_1-C_2)$$
or
$$D_{II}= 1/9(-E_1+5E_3-2C_1)\sim 1/9(2E_3-C_1+C_2).$$
}

\begin{proof} In this case, $b=DE_3=DK=2$ and $D^2=0$, thus (\ref{eqDelta}) becomes
  $$4b^2-36D^2=16=(3a-6)^2+s^2$$ for some integer $s$. Note that $16=4^2+0^2$ is the unique expression as a sum of two squares. Since $3a-6\neq \pm 4$ for any integer $a$,
  we have $$a=2,\,\,\,s=4.$$
  The quadratic equation (\ref{eqc}) yields solutions for $c$
  $$c=\Big((3b-a)\pm\sqrt{\frac{\Delta}{4}}\Big)/2=((3b-a)\pm s)/2=0 \,\,{\rm or}\,\,4.$$
  Thus we have two solutions: $(a,b,c)=(2, 2, 0)$ or $(2, 2,4)$, yielding the two solutions $D_I$ and $D_{II}$. Finally the numerical equivalences follow from the linear equivalence
  $$9E_3\equiv 3(E_1+C_1+C_2)$$ (see the subsection 1.5.)
\end{proof}

 \bigskip
{\bf Step II.}
\textit{ For any $i=1,2,3$, $D_IE_i=D_{II}E_i=2$.
$$D_IC_1=0,\,D_IC_2=4,\, D_IC_3=8,\,D_IC_4=4,\,D_IF=8;$$
$$D_{II}C_1=4,\,D_{II}C_2=0,\, D_{II}C_3=4,\,D_{II}C_4=8,\,D_{II}F=8.$$
The ${\bbQ}$-divisors $D_I$ and $D_{II}$  cannot represent simultaneously integral divisors, i.e., cannot exist simultaneously in $\NS(X)$.
}

\begin{proof} The intersection numbers can be obtained by using Table \ref{intCS}. Another computation shows that
$$D_ID_{II}=8/9,$$ not an integer, thus $D_I$ and $D_{II}$ cannot represent simultaneously integral divisors.
\end{proof}

Once the order of the 3 points $O_1, O_2, O_3$ were fixed, the curves $E_1, E_2$ are distinguished from each other as they have different multiplicities at $O_1$. But $C_1$ and $C_2$ (resp. $C_3$ and $C_4$) cannot from each other, as they have the same multiplicity at each $O_i$. One may switch the names of $C_1$ and $C_2$, and simultaneously switch the names of $C_3$ and $C_4$. Under this switch, $D_I$ and $D_{II}$ are interchanged. Thus one may assume that only $D_I$ can be represented by an integral divisor.

 \bigskip
{\bf Step III.}
\textit{ All divisors $D$, not necessarily effective,  with $K_XD=2, D^2=0$ are numerically equivalent to
$$D_I= 1/9(E_1-E_3+2C_1)\sim 1/9(2E_3+C_1-C_2).$$
Here among the two curves in $\pi^{-1}(D_A)$ with multiplicities $0,1,2$ at $O_1$, $O_2$, $O_3$, respectively the choice of $C_1$ is such that $DC_1=0$.
}

\medskip
This completes the proof of the first statement of Lemma \ref{Main2}. The following proves the second statement.

 \bigskip
{\bf Step IV.}
\textit{  For any smooth curve $B$ on $X$ with $K_XB=2, B^2=0$ the image $\sigma(B)$ under an automorphism $\sigma$ is another such curve, and is disjoint from $B$.
}

\begin{proof} First note that $O_i\notin B$ for any $i$ (if $O_i\in B$, then the intersection number $BE_j=2$ is greater than or equal to the product of the multiplicities of $B$ and $E_j$ at $O_i$ for all $j$, impossible.)

Suppose that $\sigma(B)=B$ as curves. Then, since $E_1$ is $\Aut(X)$-invariant, we have $\sigma(B\cap E_1)=B\cap E_1$. Since $BE_1=2$, the set $B\cap E_1$ is non-empty and has at most two points.
Since $\sigma$ is of order 3, it fixes a point in $B\cap E_1$, which must be $Q_i$ for some $i$. Then $Q_i\in E_1$, impossible since none of the six points $Q_1,..., Q_6$ is contained in the union of the 7 geodesic curves $E_1, E_2, E_3$, $C_1, ..., C_4$. This proves that $\sigma(B)\neq B$. They are numerically equivalent to each other by Step III, so $\sigma(B) B=B^2=0$, hence they are disjoint from each other.
\end{proof}

\section{The quotient of the Cartwright-Steger surface}

The quotient $X/Aut(X)$ has $$K_{X/Aut(X)}^2=\dfrac{1}{3}K_X^2=3$$ and has 3 singular points of type $1/3(1,1)$ at the images $\bar{O_i}$ of $O_i$ and 6 singular points of type $1/3(1,2)$ at the images $\bar{Q_j}$ of $Q_j$. Let $$\nu: Y\to X/Aut(X)$$ be the minimal resolution, $R_i$ be the $(-3)$-curve lying over the singular point $\bar{O_i}$, and $R_{j1}\--R_{j2}$ the $A_2$-cofiguration of $(-2)$-curves lying over $\bar{Q_j}$. Since $X/Aut(X)$ is simply connected by \cite{CS2}, so is $Y$.  Thus a numerical equivalence between integral divisors on $Y$ is a linear equivalence.
From the information in Section 1 one gets
$$p_g(Y):=h^{2,0}(Y)=1,\,\,\,\pi_1(Y)=\{1\},\,\,\, h^{1,0}(Y)=0.$$
Computing the adjunction, one gets
$$K_Y= \nu^*K_{X/Aut(X)}-\dfrac{1}{3}(R_1+R_2+R_3),$$
which implies that
$$K_Y^2=2.$$ This, together with N\"other formula, determines the remaining Hodge number
$$h^{1,1}(Y)=18.$$

\bigskip
{\bf Notation.}
For a curve $D$ on $X$,  the image on $X/Aut(X)$ will be denoted by $\bar{D}$ and the proper transform of $\bar{D}$ on $Y$ by $D'$.

\begin{proposition}\label{Y} \begin{enumerate} \item $K_Y=E_3'+R_2$. In particular, $K_Y$ is nef.
\item $Y$ is a simply connected minimal surface of general type with $K_Y^2=2$, $p_g(Y)=1$, $b_2(Y)=20$, $\rk Pic(Y)=h^{1,1}(Y)=18$.
\item $E_i'$ is a $(-3)$-curve for $i=1,2,3$.
\item $C_1'$ and  $C_2'$ are (smooth) elliptic curves.
\item $C_3'$ and $C_4'$ are curves with geometric genus $1$ and arithmetic genus $11$.
\item The $12$  $(-2)$-curves $R_{lm}$ are disjoint from the $10$ curves $E_i'$, $R_j$, $C_k'$.
\item The intersection matrix  of the $10$ curves is given as follows:
\begin{table}[ht]
\centering
  \begin{tabular}{|c|c|c|c|c|c|c|c|c|c|c|}\hline
&$E_1'$&$E_2'$&$E_3'$&$R_1$&$R_2$&$R_3$&$C_1'$&$C_2'$&$C_3'$&$C_4'$  \\ \hline\hline
$E_1'$&$-3$&$0$&$0$&$3$&$1$&$2$&$2$&$2$&$2$&$2$\\
$E_2'$&$0$&$-3$&$0$&$2$&$1$&$3$&$0$&$0$&$4$&$4$\\
$E_3'$&$0$&$0$&$-3$&$1$&$4$&$1$&$1$&$1$&$3$&$3$\\
$R_1$&$3$&$2$&$1$&$-3$&$0$&$0$&$0$&$0$&$4$&$4$ \\
$R_2$&$1$&$1$&$4$&$0$&$-3$&$0$&$1$&$1$&$3$&$3$ \\
$R_3$&$2$&$3$&$1$&$0$&$0$&$-3$&$2$&$2$&$2$&$2$ \\
$C_1'$&$2$&$0$&$1$&$0$&$1$&$2$&$-2$&$4$&$10$&$4$\\
$C_2'$&$2$&$0$&$1$&$0$&$1$&$2$&$4$&$-2$&$4$&$10$\\
$C_3'$&$2$&$4$&$3$&$4$&$3$&$2$&$10$&$4$&$14$&$20$\\
$C_4'$&$2$&$4$&$3$&$4$&$3$&$2$&$4$&$10$&$20$&$14$
\\\hline \hline
\end{tabular}
\caption{}\label{intY}
\end{table}
\item $\Pic(Y)=\NS(Y)$ is generated up to finite index by the $18$ curves $$E_1', E_3', R_1, R_2, R_3, C_1', R_{ij}.$$
\item There are many linear equivalences on $Y$, e.g.,
$$ E_1'+E_2'+R_1+R_3\equiv 2E_3'+2R_2\equiv 2K_Y,$$
$$ C_1'+C_3'\equiv C_2'+C_4'\equiv 4E_3'+4R_2\equiv 4K_Y,$$
$$ E_1'+C_1'+C_2'+R_3\equiv 3E_3'+3R_2\equiv 3K_Y,$$
$$ E_2'+E_3'+R_1+R_2\equiv C_1'+C_2'.$$
\item The image $F'$ on $Y$ of a Albanese fibre $F\sim -E_1+5E_2$ on $X$
$$F'\equiv =\nu^*(-3\bar{E_1}+15\bar{E_2})=-3E_1'+15E_2'+7R_1+4R_2+13R_3.$$
\item The genus $19$ fibration $|F'|$ on $Y$ over $\bbP^1$ has $3$ reducible fibres,
$$3F_0'+R_1+R_2+R_3,$$
$$3F_1'+2R_{12}+R_{11}+2R_{22}+R_{21}+2R_{32}+R_{31},$$
$$3F_2'+2R_{42}+R_{41}+2R_{52}+R_{51}+2R_{62}+R_{61}.$$
\end{enumerate}
\end{proposition}

\section{The quotient of $Y$ by Borisov involution}

The canonical ring of $Y$ is generated by 1 element in degree 1, 3 elements in degree 2,  4 elements in degree 3.

L. Borisov has informed me that he found an octic equation for $Y$ in $\bbP^4$, by first obtaining equations of a ball quotient which is a $\bbZ_7:\bbZ_3$ Galois cover of $X/\Aut(X)$ (this cover does not factor through $X$), then getting the octic as the equation of the invariant functions.  As a bi-product he found an involution $\alpha$ of $Y$, which switches the six curves
$$E_1'\leftrightarrow R_3, \,\,\, E_2'\leftrightarrow R_1, \,\,\,E_3'\leftrightarrow R_2,$$
and  permutes the six $A_2$-configurations $R_{i1}\--R_{i2}$ into 3 orbits.

In this section we prove the following:

\begin{proposition}\label{Z}  \begin{enumerate}
\item $\alpha(C_i')=C_i'$ for $i=1,2,3,4$.
 \item The involution $\alpha$ is fixed point free.
\item The quotient $Z:=Y/\alpha$ is a minimal surface of general type with $$K_Z^2=1,\,\,\, p_g(Z)=0,\,\,\, \pi_1(Z)=\bbZ/2\bbZ.$$
\end{enumerate}
\end{proposition}

 \begin{lemma}\label{twoC} $\alpha(C_1')=C_1'\,\,{\rm or}\,\, -R_3-E_1'+3(E_3'+R_2)-C_1'.$
 \end{lemma}

 \begin{proof}
 Write $\alpha(C_1)= xE_1'+yE_3'+a_1R_1+a_2R_2+a_3R_3+bC_1'+\Sigma$ with rational coefficients,  where $\Sigma$ is supported on $\cup R_{ij}$.
 Intersecting with the 6 curves $E_i', R_j$, we get five independent equations, hence the reduced form
 $$\alpha(C_1)= x(R_3+E_1')-3x(R_2+E_3')+(1+2x)C_1'+\Sigma.$$
 From  $\alpha(C_1')^2=-2$, we get $$12x^2+12x=\Sigma^2$$
 The right hand side is non-positive and the left is non-negative, so $\Sigma=0$ and $x=0$ or $-1$.
  \end{proof}

\begin{lemma} The second possibility in Lemma \ref{twoC} cannot occur.
\end{lemma}

 \begin{proof} Suppose that $\alpha(C_1')=-R_3-E_1'+3(E_3'+R_2)-C_1'.$\\ Then $\Tr\alpha|\NS(Y)=-2.$

 \bigskip
{\bf Case 1.} $\alpha=-1$ on $H^2(\calO_Y).$\\
In this case, $p_g(Z)=q(Z)=0$.
By the  topological Lefschetz fixed formula,
$$e(Y^{\alpha})=2+\Tr\alpha|H^2(Y,\bbZ)=2+\Tr\alpha|\NS(Y)+2\Tr\alpha|H^2(\calO_Y)=-2.$$ Let the fixed locus $Y^{\alpha}$ consist of $2m$ points $P_1,..., P_{2m}$ and  curves $A_1,..., A_t$.
Then $e(Y^{\alpha})= 2m+ \sum( 2-2g(A_i))=-2$, and
by Hurwitz $e(Z)=2m+10$ and $K_Z^2=2-2m.$
Note that $H^2(Z, \bbZ)=\NS(Z)$ is unimodular of rank $2m+8$.
On the other hand, $Z$ contains $2m$ $(-2)$-curves, three  $A_2$-configurations, and two curves $\bar{R_1}, \bar{R_3}$ with
$$\bar{R_1}^2=\bar{R_3}^2=-1, \,\,\,\bar{R_1}\bar{R_3}=3.$$
Thus $Z$ contains $2m+8$ curves whose intersection matrix has $|\det|=2^{2m}.3^3.8$, not a square, a contradiction!

 \bigskip
{\bf Case 2.} $\alpha=1$ on $H^2(\calO_Y).$\\
In this case, $p_g(Z)=1, q(Z)=0$.
By the  topological Lefschetz, $e(Y^{\alpha})=2.$ Let $Y^{\alpha}$ consist of $2m$ points $P_1,..., P_{2m}$ and  curves $A_1,..., A_t$.
Then $e(Y^{\alpha})= 2m+ \sum( 2-2g(A_i))=2$.
By the holomorphic Lefschetz fixed point formula (cf. \cite{K12}),
$$1-0+1=2m/4+ \sum\Big( (1-g(A_i))/2+A_i^2/4\Big)=e(Y^{\alpha})/4+\sum A_i^2/4,$$
so
$$\sum A_i^2=6,\,\,\,\sum K_YA_i=2m-8.$$
For surfaces with $p_g>0$, $|\det\NS|$ may not be a square.
We argue in a different way.
Every $\alpha$-invariant divisor such as $A=\sum A_i$, if never intersects the 6 $A_2$-configurations, can be written as
$$A=x(E_1'+R_3)+y(E_3'+R_2)$$
for some rational numbers $x, y$. In our case
$$6=A^2=-2x^2+2y^2+4xy=2(x+y)^2-4x^2,$$ $$2m-8=K_YA=2x+2y.$$ Eliminating $y$, we get $$2x^2=(m-4)^2-3.$$
Since $m$ is an integer, so is $x$, but then it is elementary to check that  this Diophantine equation has no integer solution.
\end{proof}

This, together with the linear equivalences from Proposition \ref{Y}, implies the first assertion of Proposition \ref{Z}.

\medskip
Now  since $\alpha(C_1')=C_1'$, $$\Tr\alpha|\NS(Y)=0.$$  We will prove the last two assertions of Proposition \ref{Z}.

 \bigskip
{\bf Case I.} $\alpha=-1$ on $H^2(\calO_Y).$\\
In this case, $p_g(Z)=q(Z)=0$.
By the  topological Lefschetz fixed formula,
$$e(Y^{\alpha})=2+\Tr\alpha|H^2(Y,\bbZ)=2+\Tr\alpha|\NS(Y)+2\Tr\alpha|H^2(\calO_Y)=0.$$ Let the fixed locus $Y^{\alpha}$ consist of $2m$ points $P_1,..., P_{2m}$ and  curves $A_1,..., A_t$.
Then $$e(Y^{\alpha})= 2m+ \sum( 2-2g(A_i))=0$$ and
by Hurwitz, $$e(Z)=2m+11,\,\,\,K_Z^2=1-2m.$$
By the holomorphic Lefschetz fixed point formula (cf. \cite{K12}),
$$1-0-1=2m/4+ \sum\Big( (1-g(A_i))/2+A_i^2/4\Big)=e(Y^{\alpha})/4+\sum A_i^2/4,$$
so
$$\sum A_i^2=0,\,\,\,\sum K_YA_i=2m.$$
Every $\alpha$-invariant divisor such as $A=\sum A_i$, if never intersects the 6 $A_2$-configurations, can be written as
$$A=x(E_1'+R_3)+y(E_3'+R_2)+bC_1'$$
for some rational numbers $x, y, b$. In our case
$$0=A^2=-2x^2+2y^2-2b^2+4xy+8bx+4by.$$
Note that $E_3'A=\alpha(E_3')\alpha(A)=R_2A$, $E_2'A=R_1A$ and $E_1'A=R_3A.$
Since $ (E_1'+R_3)+(E_2'+R_1)\equiv 2E_3'+2R_2$, these imply that
$$E_1'A+E_2'A=2E_3'A.$$
By Lemma \ref{isol} the intersection number $E_i'A$ is an even integer not exceeding $E_i'\alpha(E_i')$.
From these, we infer that
$$E_1'A=E_2'A=E_3'A=0\,\,\,{\rm or}\,\,\,E_1'A=E_2'A=E_3'A=2.$$
In the latter case, $-x+y+2b=3x+y=x+y+b=2$
which together with the quadratic equation has no rational solution.
In the former case, $-x+y+2b=3x+y=x+y+b=0$
which together with the quadratic equation has one solution $x=y=b=0$. This implies that $A=0$, hence $Y^{\alpha}=\emptyset$.
This completes the proof of Proposition \ref{Z} in this case.

 \bigskip
{\bf Case II.} $\alpha=1$ on $H^2(\calO_Y).$\\
By the  topological Lefschetz, $$e(Y^{\alpha})=4.$$ Let $Y^{\alpha}$ consist of $2m$ points $P_1,..., P_{2m}$ and  curves $A_1,..., A_t$.
Then $e(Y^{\alpha})= 2m+ \sum( 2-2g(A_i))=4$.
By the holomorphic Lefschetz fixed point formula (cf. \cite{K12}),
$$1-0+1=2m/4+ \sum\Big( (1-g(A_i))/2+A_i^2/4\Big)=e(Y^{\alpha})/4+\sum A_i^2/4,$$
so
$$\sum A_i^2=4,\,\,\,\sum K_YA_i=2m-8.$$
As in the previous case, $A=\sum A_i=x(E_1'+R_3)+y(E_3'+R_2)+bC_1'$
for some rational numbers $x, y, b$ and
$$4=A^2=-2x^2+2y^2-2b^2+4xy+8bx+4by.$$
As in the previous case,
$E_1'A=E_2'A=E_3'A=0$  or $2$, thus
$$-x+y+2b=3x+y=x+y+b=0$$ or  $$-x+y+2b=3x+y=x+y+b=2.$$
Either together with the quadratic equation has no rational solution.

\medskip
The following Lemma completes the proof of Proposition \ref{Z}.

\begin{lemma}\label{isol} Let $\alpha$ be an involution on a smooth surface $V$.  Let $D$ be an irreducible  curve on $V$ such that $\alpha(D)\neq D$. If $P\in \alpha(D)\cap D$ is an isolated fixed point of $\alpha$, then every branch $D'$ of $D$ at $P$ is tangent to $\alpha(D')$. In particular, if $P\in \alpha(D)\cap D$ is a transversal intersection point, then either $P\neq \alpha(P)$ or $P$ lies on a point-wise fixed curve of $\alpha$.
 \end{lemma}

\begin{proof}
 At an isolated fixed point,  $\alpha(x,y)=(-x,-y)$ in a suitable local coordinates $x, y$. So $\alpha$ preserves all tangential directions.
\end{proof}

\section{Proof of Theorem}

Suppose that the Cartwright-Steger surface $X$ contains a smooth curve $B$ with $K_XB=2, B^2=0$. Then
By Step III, $$B\sim 1/9(2E_3+C_1-C_2).$$
Let $$p:X\to X/\langle\sigma\rangle$$ be the quotient map and
$$\nu: Y\to X/\langle\sigma\rangle$$ be the minimal resolution.

By Lemma \ref{Main2}, $p_*B$ is a smooth curve away from the singular points of $X/\langle\sigma\rangle$,
$$p^*p_*B=B+\sigma(B)+\sigma^2(B).$$ Since $p_*E=3\bar{E}$ for any $\sigma$-invariant curve $E$, we see that
$$p_*B\sim \dfrac{1}{9}p_*(2E_3+C_1-C_2)=\dfrac{1}{3}(2\bar{E_3}+\bar{C_1}-\bar{C_2}),$$
thus
$$\nu^*p_*B\sim \dfrac{1}{3}\nu^*(2\bar{E_3}+\bar{C_1}-\bar{C_2})=\dfrac{1}{3}(2E_3'+C_1'-C_2'+\dfrac{2R_1+8R_2+2R_3}{3}).$$
By Proposition \ref{Z}
$$ \alpha\nu^*p_*B\sim \dfrac{1}{3}(2R_2+C_1'-C_2'+\dfrac{2E_2'+8E_3'+2E_1'}{3}).$$
Then a direct computation using Table \ref{intY} in Proposition \ref{Y} gives
$$(\nu^*p_*B)(\alpha\nu^*p_*B)=\dfrac{4}{3},$$
not an integer, a contradiction.

\section{Further discussion on the surface $Z=Y/\langle \alpha\rangle$}

The images of $E_i'$, $R_j$, $C_k'$, on $Z=Y/\langle \alpha\rangle$ will be denoted by $\bfe_i$, $\bfr_j$, $\bfc_k$ respectively. Then
$$\bfe_1=\bfr_3,\,\,\bfe_2=\bfr_1,\,\,\bfe_3=\bfr_2.$$

\begin{proposition}\label{Z2} \begin{enumerate} \item $K_Z=\bfr_2+t$ for the unique 2-torsion divisor $t$.
\item $Z$ is a minimal surface of general type with $K_Z^2=1$, $p_g(Z)=q(Z)=0$, $\pi_1(Z)=\bbZ/2\bbZ$, $\rk Pic(Z)=b_2(Z)=9$.
\item $\bfr_1$ and $\bfr_3$ are rational curves with one node, arithmetic genus $1$.
\item $\bfr_2$ is a  rational curve with two nodes, arithmetic genus $2$.
\item $\bfc_1$ and  $\bfc_2$ are (smooth) elliptic curves.
\item $\bfc_3$ and $\bfc_4$ are curves with geometric genus $1$, arithmetic genus $6$ and  $5$ nodes.
\item The intersection matrix  of the $7$ curves is given as follows:
\begin{table}[ht]
\centering
  \begin{tabular}{|c|c|c|c|c|c|c|c|}\hline
&$\bfr_1$&$\bfr_2$&$\bfr_3$&$\bfc_1$&$\bfc_2$&$\bfc_3$&$\bfc_4$  \\ \hline\hline
$\bfr_1$&$-1$&$1$&$3$&$0$&$0$&$4$&$4$ \\
$\bfr_2$&$1$&$1$&$1$&$1$&$1$&$3$&$3$ \\
$\bfr_3$&$3$&$1$&$-1$&$2$&$2$&$2$&$2$ \\
$\bfc_1$&$0$&$1$&$2$&$-1$&$2$&$5$&$2$\\
$\bfc_2$&$0$&$1$&$2$&$2$&$-1$&$2$&$5$\\
$\bfc_3$&$4$&$3$&$2$&$5$&$2$&$7$&$10$\\
$\bfc_4$&$4$&$3$&$2$&$2$&$5$&$10$&$7$
\\\hline \hline
\end{tabular}
\caption{}\label{intZ}
\end{table}
\item The three  $A_2$-configurations $\bfr_{i1}\--\bfr_{i2}$ are disjoint from the $7$ curves $\bfr_i$, $\bfc_j$.
\item $\Pic(Z)=\NS(Z)=H^2(Z,\bbZ)$ is generated up to finite index by the $9$ curves $$\bfr_1, \bfr_2, \bfc_1, \bfr_{ij},$$ whose intersection matrix has determinant $3^4$.
\item There are many numerical equivalences on $Z$, e.g.,
$$ \bfr_1+\bfr_3\sim 2\bfr_2\equiv 2K_Z,$$
$$ \bfc_1+\bfc_3\sim \bfc_2+\bfc_4\sim 4\bfr_2\equiv 4K_Z,$$
$$ \bfc_1+\bfc_2+\bfr_3\sim 3\bfr_2\sim 3K_Z,$$
$$ \bfr_1+\bfr_2\sim \bfc_1+\bfc_2.$$
\end{enumerate}
\end{proposition}

\bibliographystyle{plain}

 \end{document}